\documentclass[12pt,reqno]{amsart}
\usepackage[usenames]{color}
\usepackage{amsmath,verbatim,graphicx,epstopdf,enumerate}
\newcommand{\I}{\mathrm{i}}
\newcommand{\D}{\mathrm{d}}

\newcommand{\lb}{\left(}

\newcommand{\vp}{\varphi}
\newcommand{\ve}{\varepsilon}

\newcommand{\rb}{\right)}
\newcommand{\PD}{\partial}

\newcommand{\Bc}{\mathcal{B}}

\newcommand{\Dc}{\mathcal{D}}

\newcommand{\Fc}{\mathcal{F}}

\newcommand{\Lc}{\mathcal{L}}

\newcommand{\Rb}{\mathbb{R}}

\newcommand{\Sb}{\mathbb{S}}

\newcommand{\Lbf}{\textbf{L}^2}
\newcommand{\Hbf}{\textbf{H}}

\newcommand{\Beq}{\begin{equation}}
\newcommand{\Eeq}{\end{equation}}
\newcommand{\beq}{\begin{equation*}}
\newcommand{\eeq}{\end{equation*}}
\newcommand{\bal}{\begin{align}}
\newcommand{\eal}{\end{align}}
\renewcommand{\O}{\Omega}

\newcommand{\n}{\nabla}

\newcommand{\bp}{\begin{prob}}
	\newcommand{\ep}{\end{prob}}
\newcommand{\bpr}{\begin{proof}}
	\newcommand{\epr}{\end{proof}}
\renewcommand{\o}{\omega}

\newcommand{\bel}[1]{\begin{equation}\label{#1}}
\newcommand{\ee}{\end{equation}}

\newtheorem{theorem}{Theorem}[section]

\newtheorem{lemma}[theorem]{Lemma}
\newtheorem{proposition}[theorem]{Proposition}

\theoremstyle{definition}

\title[Partial data inverse problem for wave equation]{Determining the time dependent matrix potential in a wave equation from partial boundary data}
\author[Mishra and Vashisth]{Rohit Kumar Mishra$^\dagger$ and Manmohan Vashisth$^{\ast}$}

\address{$^\dagger$ Department of Mathematics, University of Texas at Arlington, Texas,  United States
	\newline
	\indent E-mail:{\tt \ rohit.mishra@uta.edu}; \tt \ rohittifr2011@gmail.com}
\address{$^\ast$ Beijing Computational Science Research Center, Beijing 100193, China.
	\newline
	\indent E-mail:{\tt\ mvashisth@csrc.ac.cn; \tt\  manmohanvashisth@gmail.com}}
\begin{document}
		\maketitle
	\begin{abstract}
		 		We study the inverse problem for determining the time-dependent matrix potential appearing in the wave equation. We prove the unique determination of  potential from the knowledge of solution measured on a part of the boundary.
	\end{abstract}


	
	\section{Introduction}\label{Introduction}
	Let $\Omega\subset\Rb^{n}$ for $n\geq 2$ be a bounded open set with $C^{2}$ boundary $\PD\Omega$. For $T>0$, let  $Q:=(0,T)\times\Omega$ and we denote its lateral boundary by $\Sigma:=(0,T)\times\PD\Omega$. Throughout this article,  $\textbf{H}^s(X)$ will denote the space of vector valued functions defined on $X$ with each of its component belongs to  $H^s(X)$. Similar notations will be used for other vector valued function spaces as well such as $\textbf{C}^k(X)$, $\textbf{L}^2(X)$ etc. Let $q(t,x):=\left(q_{ij}(t,x)\right)_{1\leq i,j\leq n}$, is a time-dependent matrix valued potential with each $q_{ij}\in W^{1,\infty}(Q)$ and we write this as $q\in W^{1,\infty}(Q)$. 
 For a displacement vector $\vec{u}(t,x):=\left(u_{1}(t,x),u_{2}(t,x),\cdots, u_{n}(t,x)\right)^{T}$  and a matrix valued potential  $q(t,x)$, we denote by  $\Lc_{q}$ the  following   operator
	\begin{align}\label{definition of operator}
	\Lc_{q}\vec{u}(t,x):=
	\begin{bmatrix}\vspace*{2mm}
	\Box u_{1}(t,x) +\sum_{j=1}^{n}q_{1j}(t,x)u_{j}(t,x)\\
	\Box u_{2}(t,x) +\sum_{j=1}^{n}q_{2j}(t,x)u_{j}(t,x)\\
	\vdots\\
	\Box u_{n}(t,x)+\sum_{j=1}^{n}q_{nj}(t,x)u_{j}(t,x)\\
	\end{bmatrix}
	,\  \ (t,x)\in Q
	\end{align}
	where $\Box:=\partial_{t}^{2}-\Delta_{x}$, denotes the wave operator. 
Now we consider the following initial boundary value problem:
	 	\begin{align}\label{Equation of interest}
	 \begin{cases}
	 &\Lc_{q}\vec{u}(t,x)=\vec{0},\ (t,x)\in Q\\
	&\vec{u}(0,x)=\vec{\phi},\ \partial_{t}\vec{u}(0,x)=\vec{\psi}(x), \ x\in\Omega\\
	 &\vec{u}(t,x)=\vec{f}(t,x),\ (t,x)\in \Sigma.
	 \end{cases}
	 \end{align} 
Using Theorem \ref{Exitence uniqueness theorem} in \S \ref{Wellposedness}, if for $q\in L^{\infty}(Q)$, $\vec{\phi}\in \Hbf^{1}(Q)$, $\vec{\psi}\in \textbf{L}^{2}(\Omega)$ and $\vec{f}\in \textbf{H}^{1}(\Sigma)$ is such that $\vec{f}(0,x)=\vec{\phi}(x)$ for $x\in\PD\Omega$, then there exists a unique solution $\vec{u}$ of \eqref{Equation of interest} satisfying the following
	 \begin{align*}
	 \begin{aligned}
	 \vec{u}\in \textbf{C}^{1}\lb [0,T];\textbf{L}^{2}(\Omega)\rb\cap \textbf{C}\lb [0,T];\textbf{H}^{1}(\Omega)\rb  \mbox{and}\  \PD_{\nu}\vec{u}\in \textbf{L}^{2}(\Sigma),
	 \end{aligned}
	 \end{align*}
	 where $\PD_{\nu}\vec{u}$ represents the component-wise  normal derivative of vector $\vec{u}$, that is $\PD_{\nu}\vec{u}:=  (\PD_{\nu} u_1, \cdots, \PD_{\nu}u_n)$.
	 
	 Based on this we define the continuous linear input-output operator $\Lambda_{q}:\Hbf^{1}(\Omega)\times\textbf{L}^{2}(\Omega)\times\textbf{H}^{1}(\Sigma)\rightarrow \textbf{H}^{1}(\Omega)\times\textbf{L}^{2}(\Sigma)$ by
	 \begin{align}\label{Definition of input-output operator}
	 \Lambda_{q}\lb \vec{\phi},\vec{\psi},\vec{f}\rb:=\Big( \vec{u}(T,\cdot),\PD_{\nu}\vec{u}|_{\Sigma}\Big).
	 \end{align}  
In this paper we consider the inverse problem of determining time-dependent potential $q$ from the knowledge of input-output operator $\Lambda_{q}$ measured on a subset of $\PD Q$. Our goal is to  prove a uniqueness result for determining $q$ from the partial information of $\Lambda_{q}$ measured  on $\PD Q$ (see Theorem \ref{Main Theorem} below in \S \ref{Main result statement} for more details). 
	  
	  Uniqueness issues for determining the coefficients in hyperbolic inverse problems are of great interest in last few decades. There have been extensive works in the literature regarding the identification of coefficients from boundary measurements involving the single wave equation while concerning the coefficients identification problems for the system of hyperbolic equations, not many results are available in literature. To the best of our knowledge the problem of determining the time-independent matrix potential  appearing in a one dimensional wave equation from boundary measurement  is first studied in \cite{Avdonin_Belishev_Imanov} and recently this result has been extended  in \cite{Khanfer_Bukhgeim} to the determination of matrix valued potential using finite number of boundary measurements. Following the ideas used in \cite{Blagoveshchenskii}, authors in \cite{Avdonin_Belishev_Imanov} showed that the time-independent matrix potential can be recovered from the boundary measurements.   
	  Eskin and Ralston in \cite{Eskin_Ralston_Yang_Mills} studied the problem of determining the first order as well as zeroth order time independent matrix valued perturbations in hyperbolic equations and proved the uniqueness  up to  a gauge invariance (see \cite{Eskin_Ralston}) from the full boundary measurements. The gauge invariance appears only because of first (or higher) order perturbations and hence in the present work there will be no gauge invariance since we are only considering the zeroth order perturbation. Hence one can hope to recover the matrix potential $q$ uniquely for the above system of equation \eqref{Equation of interest} from the boundary measurements and this is the question we study in the current article.
	   Next we mention the works related to the single wave equation which are closely related to the problem we study in this article. Unique determination for time independent scalar potential from boundary data appearing in \eqref{Equation of interest} is initially  studied by  Bukhge\u{i}m and Klibanov in \cite{Bukhgeuim_Klibanov_Uniqueness_1981} (see also \cite{Rakesh_Symes_Uniqueness_1988}). In \cite{Rakesh_Symes_Uniqueness_1988} uniqueness was proved using the geometric optics solutions inspired by the work of Sylvester  and Uhlmann \cite{Sylvester_Uhlmann_Calderon_problem_1987} for elliptic problem. Rakesh and Ramm  in \cite{Ramm_Sjostrand_IP_wave_equation_potential_1991} considered the  unique determination of  time-dependent scalar potential and they proved that the potential can be determined uniquely in some subset of $Q$ from the knowledge of the  Dirichlet to Neumann map measured on $\Sigma$. In \cite{Ramm_Sjostrand_IP_wave_equation_potential_1991} the wave equation with time-dependent potential in $\Rb\times\Omega$ is considered and they proved the uniqueness result for determining the coefficient from the  Dirichlet to Neumann map measured on $\Rb\times\PD\Omega$. For finite time domain $Q$ the problem for determining the time-dependent potential was studied by  \cite{Isakov_Completeness_Product_solutions_IP_1991} where uniqueness result was proved using  informations of the solutions at initial and final time in addition to the Dirichlet to Neumann map. Recently Kian in \cite{Kian} proved that the uniqueness considered in \cite{Isakov_Completeness_Product_solutions_IP_1991} can be shown using the less information than that of \cite{Isakov_Completeness_Product_solutions_IP_1991}. Using the Carleman estimate together with geometric optic solutions Kian in \cite{Kian}  established the uniqueness for scalar time dependent potential using the informations of solution measured on a suitable subset of $\PD Q$. For anisotropic  wave equation the  unique determination for the time-dependent scalar potential from partial boundary data has been considered in \cite{Kian_Oksanen_anisotropic_wave_potential}. For more works related to the determination of coefficients appearing in the single wave equation from boundary measurements, we refer to \cite{Anikonov_Cheng_Yamamoto,Belishev_BC_Method_2011,Bellassoued_Jellali_Yamamoto_Lipschitz_stability_hyperbolic,Bellassoued_Yamamoto,Bellassoued_Jellali_Yamamoto_stability_hyperbolic,Ibtissem_Stability_potential,Katchalov_Kurylev_Lassas_Book_2001,Kian_Stability_potential_partial_data,Kian,Stefanov-Yang,Stefanov_Inverse_scattering_potential_time_dependent_1989} and references therein.

	In this paper we consider the unique determination of time-dependent matrix valued potential $q(t,x)$  appearing in \eqref{Equation of interest} from the partial boundary data.  Our work can be seen as an extension of the work of \cite{Kian} who considered the aforementioned problem for determining the scalar time-dependent potential $q$ appearing in \eqref{Equation of interest}.

	The paper is organized as follows. In \S \ref{Wellposedness} we prove the well-posedness of the forward problem for Equation \eqref{Equation of interest}. In \S \ref{Main result statement}, we state the main result of the article. \S \ref{Carleman Estimate} is devoted to derive the Carleman estimates which will be used to prove the existence of geometric optics (GO) solutions and in \S \ref{GO solutions}, we construct the required GO solutions. Finally in  \S \ref{proof main result},  we prove the main theorem \ref{Main Theorem} of the article.
	
	\section{Preliminary result}\label{Wellposedness}

In this section we prove the existence and uniqueness for the initial boundary value problem. In particular we prove the following theorem: 
\begin{theorem}\label{Exitence uniqueness theorem}
	Let $q\in W^{1,\infty}(Q)$ be a time-dependent  matrix potential. Suppose $\vec{\phi}\in \Hbf^{1}(\Omega)$, $\vec{\psi}\in \textbf{L}^{2}(\Omega)$ and  $\vec{f}\in \Hbf^{1}(\Sigma)$ is such that $\vec{f}(0,x)=\vec{\phi}(x)$ for $x\in\PD\Omega$.  Then there exists a unique solution $\vec{u}$ to \eqref{Equation of interest} satisfying the following $$\vec{u}\in \textbf{C}^{1}\lb [0,T];\textbf{L}^{2}(\Omega)\rb\cap \textbf{C}\lb [0,T];\textbf{H}^{1}(\Omega)\rb  \mbox{and}\  \PD_{\nu}\vec{u}\in \textbf{L}^{2}(\Sigma).$$ 
	Moreover, there exists a constant $C>0$ depending only on $q$, $T$ and $\Omega$ such that
	\begin{align}\label{Estimate for solution}
	\lVert\PD_{\nu}\vec{u}\rVert_{\textbf{L}^{2}(\Sigma)}+\lVert\vec{u}\rVert_{\textbf{H}^{1}(Q)}\leq C \lb\lVert\vec{\phi}\rVert_{\textbf{H}^{1}(\Omega)}+ \lVert\vec{\psi}\rVert_{\textbf{L}^{2}(\Omega)}+\lVert\vec{f}\rVert_{\textbf{L}^{2}(\Sigma)}\rb
	\end{align}
	holds. 
	\begin{proof}
Let us write the solution $\vec{u}$ to \eqref{Equation of interest} into two terms as $\vec{u}(t,x):=\vec{v}(t,x)+\vec{w}(t,x)$ where $\vec{v}$ is solution to 

			\begin{equation}\label{Equation for v preliminary}
		\begin{aligned}
		\begin{cases}
		&\PD_{t}^{2}\vec{v}(t,x)-\Delta_{x}\vec{v}(t,x)=\vec{0},\ (t,x)\in Q\\
		&\vec{v}(0,x)=\vec{\phi}(x),\ \PD_{t}\vec{v}(0,x)=\vec{\psi}(x),\ x\in\Omega\\
		&\vec{v}(t,x)=\vec{f}(t,x),\ (t,x)\in\Sigma
		\end{cases}
		\end{aligned}
		\end{equation}
and $\vec{w}$ is solution to 
			\begin{equation}\label{Equation for w preliminary}
		\begin{aligned}
		\begin{cases}
		&\mathcal{L}_{q}\vec{w}(t,x)=-q(t,x)\vec{v}(t,x),\ (t,x)\in Q\\
		&\vec{w}(0,x)=\PD_{t}\vec{w}(0,x)=\vec{0},\ x\in\Omega\\
		&\vec{w}(t,x)=\vec{0},\ (t,x)\in\Sigma.
		\end{cases}
		\end{aligned}
		\end{equation}
	Since Equation \eqref{Equation for v preliminary} is a decoupled system of wave equations therefore following  Theorem $2.30$ in \cite{Katchalov_Kurylev_Lassas_Book_2001} there exists a unique solution $\vec{v}(t,x)$ to \eqref{Equation for v preliminary} such that 
	\begin{align}\label{Estimate for solution v preliminary}
	\begin{aligned}
	&\vec{v}\in \textbf{C}^{1}\lb [0,T];\textbf{L}^{2}(\Omega)\rb\cap \textbf{C}\lb [0,T];\textbf{H}^{1}(\Omega)\rb  \mbox{and}\  \PD_{\nu}\vec{v}\in \textbf{L}^{2}(\Sigma)\\
	\mbox{and} \ & 	
	\lVert\PD_{\nu}\vec{v}\rVert_{\textbf{L}^{2}(\Sigma)}+\lVert\vec{v}\rVert_{\textbf{H}^{1}(Q)}\leq C \lb\lVert\vec{\phi}\rVert_{\textbf{H}^{1}(\Omega)}+ \lVert\vec{\psi}\rVert_{\textbf{L}^{2}(\Omega)}+\lVert\vec{f}\rVert_{\textbf{L}^{2}(\Sigma)}\rb
	\end{aligned}
	\end{align}
	holds for some constant $C>0$ independent of $\vec{v}$. Using Equation \eqref{Estimate for solution v preliminary} and the fact that $q\in W^{1,\infty}(Q)$, we have $q\vec{v}\in \textbf{L}^{2}(Q)$.
Now following the arguments from \cite{Katchalov_Kurylev_Lassas_Book_2001,Lions_Magenese_Book,Lions} we prove the existence and uniqueness for $\vec{w}$ solution to \eqref{Equation for w preliminary}.  We define the time-dependent bilinear form $a(t;\cdot,\cdot)$ on $\textbf{H}^{1}_{0}(\Omega)$  by 
\begin{equation}\label{Definition of bilinear form}
a(t;\vec{h},\vec{g}):=\int\limits_{\Omega}\lb \nabla_{x}\vec{h}(x)\cdot\overline{\nabla_{x}\vec{g}(x)}+q(t,x)\vec{h}(x)\cdot\overline{\vec{h}(x)}\rb\D x,\ \mbox{for}\ \vec{h},\vec{g}\in \textbf{H}^{1}_{0}(\Omega).
\end{equation}
Since $\vec{h},\vec{g}$ are time-independent and $q\in L^{\infty}(Q)$ therefore for each fixed $\vec{h},\vec{g}\in \textbf{H}^{1}_{0}(\Omega)$ we have $a(t;\vec{h},\vec{g})\in L^{\infty}(0,T)$. Also using the Cauchy-Schwartz inequality and the fact that $q\in L^{\infty}(Q)$ we get
\begin{align}\label{Continuity of a}
\lvert a(t;\vec{h},\vec{g})\rvert \leq C\lVert \vec{h}\rVert_{\textbf{H}^{1}_{0}(\Omega)}\lVert \vec{g}\rVert_{\textbf{H}^{1}_{0}(\Omega)}
\end{align}
where constant $C>0$ is independent of $\vec{h}$ and $\vec{g}$. Next  consider 
\begin{align*}
\begin{aligned}
\lvert a(t;\vec{h},\vec{h})\rvert&= \Big\lvert \int\limits_{\Omega}\lb \lvert{\nabla_{x}\vec{h}}(x)\rvert^{2}+q(t,x)\vec{h}(x)\cdot\overline{\vec{h}(x)}\rb\D x\Big\rvert\\
& \quad \geq \lVert \nabla_{x}\vec{h}\rVert^{2}_{\textbf{L}^{2}(\Omega)}-\lVert q\rVert_{L^{\infty}(Q)}\lVert \vec{h}\rVert^{2}_{\textbf{L}^{2}(\Omega)}.
\end{aligned}
\end{align*}
Choosing $\lambda  >\lVert {q}\rVert_{L^{\infty}(Q)}$ in above equation, we get
\begin{align}\label{Coercivity fo a}
\begin{aligned}
\lvert a(t;\vec{h},\vec{h})\rvert+\lambda \lVert \vec{h}\rVert^{2}_{\textbf{L}^{2}(\Omega)} \geq \alpha \lVert \vec{h}\rVert^{2}_{\textbf{H}^{1}(\Omega)},\ \mbox{for some constant $\alpha>0$}.
\end{aligned}
\end{align}	
Combining Equations \eqref{Definition of bilinear form}, \eqref{Continuity of a} and \eqref{Coercivity fo a}, we get that $t\mapsto a(t;\vec{h},\vec{g})$ is continuous  
 bilinear  form for all $\vec{h},\vec{g}\in \textbf{H}^{1}_{0}(\Omega)$. Also note that the principle part of $a(t;\cdot,\cdot)$ given by 
\begin{align}\label{Principle part of a}
a(t;\vec{h},\vec{g})=\int\limits_{\Omega} \nabla_{x}\vec{h}(x)\cdot\overline{\nabla_{x}\vec{g}(x)}\D x
\end{align} 
is anti-symmetric.  Therefore using Theorem $8.1$ together with Remark $8.1$ of Chapter $3$  in \cite{Lions_Magenese_Book} (see also \cite{Lions}), we have that the initial boundary value problem given by \eqref{Equation for w preliminary} admits a unique solution $\vec{w}\in \textbf{C}^{1}\lb [0,T];\textbf{L}^{2}(\Omega)\rb \cap \textbf{C}\lb [0,T];\textbf{H}^{1}(\Omega)\rb$ and it satisfies the following estimate 
\begin{align}\label{Energy estimate for w}
\begin{aligned}
\int\limits_{Q}\lb \lvert\vec{w}(t,x)\rvert^{2}+\lvert\PD_{t}\vec{w}(t,x)\rvert^{2}+\lvert\nabla_{x}\vec{w}(t,x)\rvert^{2}\rb \D x\D t\leq C\lb \lVert\vec{\phi}\rVert_{\textbf{H}^{1}(\Omega)}+ \lVert\vec{\psi}\rVert_{\textbf{L}^{2}(\Omega)}+\lVert\vec{f}\rVert_{\textbf{L}^{2}(\Sigma)}\rb. 
\end{aligned}
\end{align}
 Next we prove that $\PD_{\nu}\vec{w}\in \textbf{L}^{2}(\Sigma)$. 
We follow the arguments similar to the one used in \cite{Lasiecka_Lions_Triggiani_Nonhomogeneous_BVP_hyperbolic_1986} for the wave equation with scalar potential. Let $\nu(x)$ denote the outward unit normal to $\PD\Omega$ at $x\in\PD\Omega$. We extend this to $\overline{\Omega}$ and denote the extended one by $\nu(x)$ itself.  Now consider the following  integral 
\begin{align*}
\begin{aligned}
&\int\limits_{Q}\lb T-t\rb \Lc_{q}\vec{w}(t,x)\cdot \lb \nu(x)\cdot\nabla_{x}\vec{w}(t,x)\rb\D x \D t =\int\limits_{Q}\lb T
-t\rb\PD_{t}^{2}\vec{w}(t,x)\cdot \lb \nu(x)\cdot\nabla_{x}\vec{w}(t,x)\rb\D x \D t\\
&\ \ \ \  -\int\limits_{Q}\lb T-t\rb\Delta_{x}\vec{w}(t,x)\cdot \lb \nu(x)\cdot\nabla_{x}\vec{w}(t,x)\rb\D x \D t +\int\limits_{Q}\lb T-t\rb q(t,x)\vec{w}(t,x)\cdot \lb \nu(x)\cdot\nabla_{x}\vec{w}(t,x)\rb\D x \D t\\
&=\sum_{j=1}^{n}\int\limits_{Q}\lb T-t\rb\PD_{t}^{2}w_{j}(t,x)\lb \nu(x)\cdot\nabla_{x}w_{j}(t,x)\rb\D x\D t-\sum_{j=1}^{n}\int\limits_{Q}\lb T-t\rb\Delta_{x}w_{j}(t,x)\lb \nu(x)\cdot\nabla_{x}w_{j}(t,x)\rb\D x\D t\\
&\ \ \ +\sum_{i,j=1}^{n}\int\limits_{Q}\lb T-t\rb q_{ij}(t,x)w_{j}(t,x)\lb \nu(x)\cdot\nabla_{x}{w_{j}}(t,x)\rb\D x\D t:= A_{1}+A_{2}+A_{3}
\end{aligned}
\end{align*}
where \begin{align*}
\begin{aligned}
&A_{1}:= \sum_{j=1}^{n}\int\limits_{Q}\lb T-t\rb\PD_{t}^{2}w_{j}(t,x)\lb \nu(x)\cdot\nabla_{x} w_{j}(t,x)\rb\D x\D t\\
&A_{2}:=-\sum_{j=1}^{n}\int\limits_{Q}\lb T-t\rb\Delta_{x}w_{j}(t,x)\lb \nu(x)\cdot\nabla_{x}w_{j}(t,x)\rb\D x\D t\\
&A_{3}:= \sum_{i,j=1}^{n}\int\limits_{Q}\lb T-t\rb q_{ij}(t,x)w_{j}(t,x)\lb \nu(x)\cdot\nabla_{x}{w_{i}}(t,x)\rb\D x\D t.
\end{aligned}
\end{align*}
Using Equation \eqref{Equation for w preliminary}, we have
\begin{align}\label{Sum of A1, A2 and A3}
A_{1}+A_{2}+A_{3}=-\int\limits_{Q}(T-t)q(t,x)\vec{v}(t,x)\cdot\lb\nu(x)\cdot\nabla_{x}\vec{w}(t,x)\rb\D x\D t.
\end{align}
We simplify each of $A_{j}$ for $1\leq j\leq 3$. Using integration parts, we have $A_{1}$ is 
\begin{align*}
\begin{aligned}
A_{1}&=-T\sum_{j=1}^{n}\int\limits_{\Omega} \PD_{t}w_{j}(0,x)\lb \nu(x)\cdot\nabla_{x}w_{j}(0,x)\rb\D x+\sum_{j=1}^{n}\int\limits_{Q}\PD_{t}w_{j}(t,x)\lb \nu(x)\cdot\nabla_{x}w_{j}(t,x)\rb \D x\D t\\
&\ \ \ \ \ \ \ \ \ \ -\sum_{j=1}^{n}\int\limits_{Q}\lb T-t\rb\PD_{t}w_{j}(t,x)\lb \nu(x)\cdot\nabla_{x}\PD_{t}w_{j}(t,x)\rb \D x\D t\\
&=-T\int\limits_{\Omega} \PD_{t}\vec{w}(0,x)\cdot\lb \nu(x)\cdot\nabla_{x}\vec{w}(0,x)\rb\D x+\int\limits_{Q}\PD_{t}\vec{w}(t,x)\cdot\lb \nu(x)\cdot\nabla_{x}\vec{w}(t,x)\rb\D x\D t\\
&\ \ \ \ \ \ \ \ \ -\int\limits_{Q}\frac{ T-t}{2} \nabla_{x}\cdot\lb \nu(x)\lvert\PD_{t}\vec{w}(t,x)\rvert^{2}\rb\D x\D t +\int\limits_{Q}\frac{T-t}{2}\lvert\PD_{t}\vec{w}(t,x)\rvert^{2}\nabla_{x}\cdot\nu(x)\D x \D t.
\end{aligned}
\end{align*}
Using the Gauss divergence theorem and the fact that $\vec{w}|_{\Sigma}=\vec{w}|_{t=0}=\PD_{t}\vec{w}|_{t=0}=0$, we get 
\begin{align}\label{Expression for A1}
\begin{aligned}
A_{1}&=\int\limits_{Q}\PD_{t}\vec{w}(t,x)\cdot\lb \nu(x)\cdot\nabla_{x}\vec{w}(t,x)\rb\D x\D t +\int\limits_{Q}\frac{T-t}{2}\lvert\PD_{t}\vec{w}(t,x)\rvert^{2}\nabla_{x}\cdot\nu(x)\D x \D t.
\end{aligned}
\end{align}\
Now using the integration by parts in the expression for $A_{2}$, we have
\begin{align*}
\begin{aligned}
A_{2}&=-\sum_{j=1}^{n}\int\limits_{Q}\lb T-t\rb\Delta_{x}w_{j}(t,x)\lb \nu(x)\cdot\nabla_{x}w_{j}(t,x)\rb\D x\D t\\
&=-\sum_{j=1}^{n}\int\limits_{Q}(T-t)\sum_{k,l=1}^{n}\PD_{k}^{2}w_{j}(t,x)\nu_{l}(x)\PD_{l}w_{j}(t,x)\D x\D t\\
&=-\sum_{j=1}^{n}\int\limits_{Q}(T-t)\nabla_{x}\cdot \lb \nabla_{x}w_{j}(t,x)\nu(x)\cdot\nabla_{x}w_{j}(t,x)\rb\D x\D t-\int\limits_{Q}\frac{T-t}{2}\nabla_{x}\cdot\nu(x)\lvert\nabla_{x}\vec{w}(t,x)\rvert^{2}\D x\D t\\
&\ \ \  +\sum_{j=1}^{n}\int\limits_{Q}(T-t)\sum_{k,l=1}^{n}\PD_{k}w_{j}(t,x)\PD_{k}\nu_{l}(x)\PD_{l}w_{j}(t,x)\D x\D t+\int\limits_{Q}\frac{T-t}{2}\nabla_{x}\cdot\lb \nu(x)\lvert\nabla_{x}\vec{w}(t,x)\rvert^{2}\rb\D x \D t.\\
\end{aligned}
\end{align*}
Gauss divergence theorem and $\vec{u}|_{\Sigma}=0$, gives 
\begin{equation}\label{Expression for A2}
\begin{aligned}
A_{2}&=-\int\limits_{\Sigma}\frac{T-t}{2}\lvert\PD_{\nu}\vec{w}(t,x)\rvert^{2}\D S_{x}\D t+\sum_{j=1}^{n}\int\limits_{Q}(T-t)\sum_{k,l=1}^{n}\PD_{k}w_{j}(t,x)\PD_{k}\nu_{l}(x)\PD_{l}w_{j}(t,x)\D x\D t\\
&\ \ \ \ \ \ \ -\int\limits_{Q}\frac{T-t}{2}\nabla_{x}\cdot\nu(x)\lvert\nabla_{x}\vec{w}(t,x)\rvert^{2}\D x\D t
\end{aligned}
\end{equation}
Finally using Equations \eqref{Expression for A1}, \eqref{Expression for A2} and the Cauchy-Schwartz inequality in \eqref{Sum of A1, A2 and A3}, we get 
\begin{equation*}
\begin{aligned}
\Big\lvert\int\limits_{\Sigma}\frac{T-t}{2}\lvert\PD_{\nu}\vec{w}(t,x)\rvert^{2}\D S_{x}\D t\Big\rvert\leq C\int\limits_{Q}\lb\lvert\vec{v}(t,x)\rvert^{2}+\lvert\vec{w}(t,x)\rvert^{2}+ \lvert\PD_{t}\vec{w}(t,x)\rvert^{2}+\lvert\nabla_{x}\vec{w}(t,x)\rvert^{2}\rb\D x\D t.
\end{aligned}
\end{equation*}
Hence using \eqref{Estimate for solution v preliminary} and \eqref{Energy estimate for w} in the above equation, we get 
\begin{align*}
\Big\lvert \int\limits_{\Sigma}\frac{T-t}{2}\lvert\PD_{\nu}\vec{w}(t,x)\rvert^{2}\D S_{x}\D t\Big\rvert\leq C\lb \lVert\vec{\psi}\rVert_{\textbf{L}^{2}(\Omega)}+\lVert\vec{f}\rVert_{\textbf{L}^{2}(\Sigma)}\rb. 
\end{align*}
Thus, we have  shown the following
	\begin{align}\label{Estimate for solution w preliminary}
\begin{aligned}
&\vec{w}\in \textbf{C}^{1}\lb [0,T];\textbf{L}^{2}(\Omega)\rb\cap \textbf{C}\lb [0,T];\textbf{H}^{1}(\Omega)\rb  \mbox{and}\  \PD_{\nu}\vec{w}\in \textbf{L}^{2}(\Sigma)\\
\mbox{and} \ & 	
\lVert\PD_{\nu}\vec{w}\rVert_{\textbf{L}^{2}(\Sigma)}+\lVert\vec{w}\rVert_{\textbf{H}^{1}(Q)}\leq C \lb\lVert\vec{\phi}\rVert_{\textbf{H}^{1}(\Omega)}+ \lVert\vec{\psi}\rVert_{\textbf{L}^{2}(\Omega)}+\lVert\vec{f}\rVert_{\textbf{L}^{2}(\Sigma)}\rb.
\end{aligned}
\end{align}
 Now combining Equations \eqref{Estimate for solution v preliminary} and \eqref{Estimate for solution w preliminary}, we get 
 	\begin{align*}
 \begin{aligned}
 &\vec{u}\in \textbf{C}^{1}\lb [0,T];\textbf{L}^{2}(\Omega)\rb\cap \textbf{C}\lb [0,T];\textbf{H}^{1}(\Omega)\rb  \mbox{and}\  \PD_{\nu}\vec{u}\in \textbf{L}^{2}(\Sigma)\\
 \mbox{and} \ & 	
 \lVert\PD_{\nu}\vec{u}\rVert_{\textbf{L}^{2}(\Sigma)}+\lVert\vec{u}\rVert_{\textbf{H}^{1}(Q)}\leq C \lb\lVert\vec{\phi}\rVert_{\textbf{H}^{1}(\Omega)}+ \lVert\vec{\psi}\rVert_{\textbf{L}^{2}(\Omega)}+\lVert\vec{f}\rVert_{\textbf{L}^{2}(\Sigma)}\rb.
 \end{aligned}
 \end{align*}
 This completes the proof of Theorem \ref{Exitence uniqueness theorem}. 
		\end{proof}
\end{theorem}
	\section{Statement of the main result}\label{Main result statement}
	Before stating the main result of this article, we introduce some notation. Following \cite{Bukhgeim_Uhlmann_Calderon_problem_partial_Cauchy_data_2002},  for
	fix  $\omega_{0}\in\mathbb{S}^{n-1}$ and define 
	\begin{align*}
	\partial\Omega_{+,\omega_{0}}:=\left\{x\in\partial\Omega:\ \nu(x)\cdot\omega_{0}\geq 0 \right\},\ \ \partial\Omega_{-,\omega_{0}}:=\left\{x\in\partial\Omega:\ \nu(x)\cdot\omega_{0}\leq 0 \right\}
	\end{align*}
 where $\nu(x)$ is outward unit normal to $\partial\Omega$ at $x\in\partial\Omega$. Corresponding to $\partial\Omega_{\pm,\omega_{0}}$, we denote the lateral boundary parts by 
	$\Sigma_{\pm,\omega_{0}}:=(0,T)\times\partial\Omega_{\pm,\omega_{0}}$. We denote by $F=(0,T)\times F'$ and $G=(0,T)\times G'$ where $F'$ and $G'$ are small enough open neighbourhoods of $\partial\Omega_{+,\omega_{0}}$ and  $\partial\Omega_{-,\omega_{0}}$ respectively in $\partial\Omega$. Now let $\vec{u}$ be the solution to Equation \eqref{Equation of interest} with $\vec{\phi}\in \Hbf^{1}(\Omega)$, $\vec{\psi}\in  \Lbf(\Omega)$ and $\vec{f}\in \Hbf^1(\Sigma)$ such that $\vec{f}(0,x)=\vec{\phi}(x)$ for $x\in\PD\Omega$. 
Next using Theorem \ref{Exitence uniqueness theorem}, we can define our continuous linear input-output operator $\widetilde{\Lambda_{q}}:\Hbf^{1}(\Omega)\times \Lbf(\Omega)\times \Hbf^{1}(\Sigma)\ \rightarrow  \Hbf^{1}(\Omega)\times\Lbf(G)$ given by
	\begin{align}\label{Definition of restricted input-output operator}
	\widetilde{\Lambda_{q}}(\vec{\phi},\vec{\psi},\vec{f})=\Big(\vec{u}|_{t=T},\partial_{\nu}\vec{u}|_{G}\Big)
	\end{align} 
	where  $\vec{u}$ is the solution to \eqref{Equation of interest}.	
In this paper, our aim is to prove the following uniqueness result for determining $q$  from the knowledge of $\widetilde{\Lambda}_{q}$.
	\begin{theorem}\label{Main Theorem}
		Let $q^{(1)}(t,x)$ and $q^{(2)}(t,x)$ be two sets of potentials 
		such that the components of each $q^{(i)}$ are in $W^{1,\infty}(Q)$ for $i=1,2$. Let $\vec{u}^{(i)}$ be solutions to \eqref{Equation of interest} when $q=q^{(i)}$ and  $\widetilde{\Lambda}_{q^{(i)}}$ for $i=1,2$ be the input-output operators  defined by \eqref{Definition of input-output operator} corresponding to $\vec{u}^{(i)}$. If
		\begin{align}\label{Equality of input-output operator}
		\widetilde{\Lambda}_{q^{(1)}}(\vec{\phi},\vec{\psi},\vec{f})=\widetilde{\Lambda}_{q^{(2)}}(\vec{\phi},\vec{\psi},\vec{f}),\ \text{for} \ (\vec{\phi},\vec{\psi},\vec{f})\in \Hbf^{1}(\Omega)\times\Lbf(\Omega)\times \Hbf^{1}(\Sigma), 
		\end{align}
		then 
		\begin{align*}
		q^{(1)}(t,x)=q^{(2)}(t,x),\ (t,x)\in Q.
		\end{align*}
	\end{theorem}
To the best of our knowledge the problem considered here has not been studied and infact  this is the first result which deals with the determination of time-dependent matrix valued coefficients appearing in hyperbolic partial differential equations from the boundary measurements. Theorem \ref{Main Theorem}, can be proved by using the Carleman estimate together with constructing  the geometric optics solutions for the wave equation with matrix valued potential. For time dependent scalar potential case this approach for hyperbolic inverse problems first appeared in \cite{Kian_Stability_potential_partial_data,Kian} and recently this approach has been used in \cite{Bellassoued_Rassas,Hu_Kian_Singular_potential_determination,Kian_Damping_partial_data_2016,Kian_Oksanen_anisotropic_wave_potential,Krishnan_Vashisth_Relativistic} for determining the coefficients in the single wave equations. To prove Theorem \ref{Main Theorem}, we follow the arguments similar to \cite{Bellassoued_Rassas,Kian_Stability_potential_partial_data,Kian}.

%
%
	\section{Carleman Estimate}\label{Carleman Estimate}
		The present section is devoted to deriving a Carleman estimate for \eqref{Equation of interest} involving the boundary terms and it will be used to control the boundary terms over subsets of the boundary where measurements are not available. In order to state the Carleman estimate, first we will fix some notation.
	For $\vec{v}=\lb v_{1},v_{2},v_{3},\cdot\cdots,v_{n}\rb^{T}\in \Hbf^{1}(Q)$, we define the $L^{2}$ norm of $\vec{v}$ by
	\[\lVert \vec{v}\rVert_{\Lbf(Q)}:=\lb\sum_{j=1}^{n}\int\limits_{Q}\lvert v_{j}(t,x)\rvert^{2}\D x\D t\rb^{1/2}=\lb\sum_{j=1}^{n}\lVert v_{j}\rVert_{L^{2}(Q)}^{2}\rb^{1/2}\]
	and 
	\[\n_{x}\vec{v}:=\left(\n_{x}v_{1},\n_{x}v_{2},\n_{x}v_{3},\cdots,\n_{x}v_{n}\right)^{T}\ \text{and}\  \omega\cdot\n_{x}\vec{v}:=\lb\omega\cdot\n_{x}v_{1},\omega\cdot\n_{x}v_{2},\cdots,\omega\cdot\n_{x}v_{n}\rb^{T}.\]
	\begin{theorem}\label{Boundary Carleman estimate}
		Let $\vp(t,x):=t+x\cdot\omega$, where $\omega\in \Sb^{n-1}$ is fixed and $q\in L^{\infty}(Q)$. Then the  Carleman estimate 
		\begin{align}\label{Boundary carleman estimate}
		\begin{aligned}
		&\lVert e^{-\vp/h}\vec{u}\rVert^{2}_{\Lbf (Q)}+h\lb  e^{-\vp/h}\partial_{\nu}\vp\partial_{\nu}\vec{u},e^{-\vp/h}\partial_{\nu}\vec{u}\rb_{\Lbf(\Sigma_{+,\omega})}+h\lb e^{-\vp(T,\cdot)/h}\partial_{t}\vec{u}(T,\cdot),e^{-\vp(T,\cdot)/h}\partial_{t}\vec{u}(T,\cdot)\rb_{\Lbf(\Omega)}\\
		&\leq C\Bigg(\lVert he^{-\vp/h}\Lc_{q} \vec{u}\rVert^{2}_{\Lbf(Q)}+\frac{1}{h}\lb e^{-\vp(T,\cdot)/h}u(T,\cdot),e^{-\vp(T,\cdot)/h}\vec{u}(T,\cdot)\rb_{\Lbf(\Omega)}\\
		&\quad \quad+h\lb e^{-\vp(T,\cdot)/h}\nabla_{x}\vec{u}(T,\cdot),e^{-\vp(T,\cdot)/h}\nabla_{x}\vec{u}(T,\cdot)\rb_{\Lbf(\Omega)}+h\lb  e^{-\vp/h}\lb-\partial_{\nu}\vp\rb\partial_{\nu}\vec{u},e^{-\vp/h}\partial_{\nu}\vec{u}\rb_{\Lbf(\Sigma_{-,\omega})}\Bigg)
		\end{aligned}
		\end{align}
		holds for all $\vec{u}\in \textbf{C}^{\ 2}(Q)$ with
		\begin{align*}
		\vec{u}|_{\Sigma}=0,\ \vec{u}|_{t=0}=\partial_{t}\vec{u}|_{t=0}=0,
		\end{align*} 
		and $h$ small enough.
		\end{theorem}
	\begin{proof}
	Define the conjugated operator  $\Box_{\vp}$ by 
		\begin{equation}\label{Conjugated Box}
		\Box_{\vp}:=h^{2}e^{-\vp/h}\Box e^{\vp/h}.
		\end{equation}
		For $\vec{v}\in \textbf{C}^{2}(Q)$, we have 
			\[
		\Box_{\vp}\vec{v}(t,x)= h^{2}\Box \vec{v}(t,x) +2h\lb  \PD_{t}-\omega\cdot \n_{x}\rb \vec{v}(t,x):=P_{1}\vec{v}(t,x)+P_{2}\vec{v}(t,x)
		\]
		where 
			\begin{align*}
		\begin{aligned}
		P_{1}\vec{v}(t,x)=h^{2}\Box   \vec{v}(t,x)\ \ 
		\text{and} \ \ 
		P_{2}\vec{v}(t,x)=2h\lb  \PD_{t}-\omega\cdot \n_{x}\rb \vec{v}(t,x).
		\end{aligned}
		\end{align*}
	Now $\Lbf$ norm of $\Box_{\vp}\vec{v}$ for $\vec{v}\in \textbf{C}^{2}(Q)$ satisfying $\vec{v}|_{\Sigma}=\vec{v}|_{t=0}=\PD_{t}\vec{v}|_{t=0}=0$, can be estimated as 
		\begin{align*}
		\begin{aligned}
		\int\limits_{Q}\lvert \Box_{\vp}\vec{v}(t,x)\rvert^{2}\D x \D t&=\int\limits_{Q}\lvert P_{1}\vec{v}(t,x)\rvert^{2}\D x\D t+\int\limits_{Q}\lvert P_{2}\vec{v}(t,x)\rvert^{2}\D x\D t + 2\int\limits_{Q} \mbox{Re}\lb P_{1}\vec{v}(t,x)\cdot\overline{P_{2}\vec{v}(t,x)} \rb \D x \D t\\
		&\geq \int\limits_{Q}\lvert P_{2}\vec{v}(t,x)\rvert^{2}\D x\D t + 2\int\limits_{Q} \mbox{Re}\lb P_{1}\vec{v}(t,x)\cdot\overline{P_{2}\vec{v}(t,x)} \rb \D x \D t\\
		&=4h^{2}\int\limits_{Q}\lvert\left(\partial_{t}-\omega\cdot\n_{x}\right)\vec{v}(t,x)\rvert^{2}\D x\D t+4h^{3}\int\limits_{Q}\mbox{Re}\lb \Box \vec{v}(t,x)\cdot\overline{\partial_{t}\vec{v}(t,x)}\rb \D x \D t\\
		&\quad \quad -4h^{3}\int\limits_{Q}\mbox{Re}\lb \Box \vec{v}(t,x)\cdot\lb\omega\cdot\overline{\nabla_{x}\vec{v}(t,x)}\rb\rb \D x \D t\\
		&=4h^{2}\sum_{j=1}^{n}\int\limits_{Q}\lvert\left(\partial_{t}-\omega\cdot\n_{x}\right)v_{j}(t,x)\rvert^{2}\D x\D t+4h^{3}\sum_{j=1}^{n}\int\limits_{Q}\mbox{Re}\lb \Box v_{j}(t,x)\overline{\partial_{t}v_{j}(t,x)}\rb \D x \D t\\
		&\quad \quad -4h^{3}\sum_{j=1}^{n}\int\limits_{Q}\mbox{Re}\lb \Box v_{j}(t,x)\lb\omega\cdot\overline{\nabla_{x}v_{j}(t,x)}\rb\rb \D x \D t\\
		&:=\sum_{j=1}^{n}\lb I_{1,j}+I_{2,j}+I_{3,j}\rb,
		\end{aligned}
		\end{align*}
		where 
		\begin{align}\label{Definition of I1j I2j I3j}
		\begin{aligned}
		&I_{1,j}:=4h^{2}\int\limits_{Q}\lvert\left(\partial_{t}-\omega\cdot\n_{x}\right)v_{j}(t,x)\rvert^{2}\D x\D t\\
		&I_{2,j}:=4h^{3}\int\limits_{Q}\mbox{Re}\lb \Box v_{j}(t,x)\overline{\partial_{t}v_{j}(t,x)}\rb \D x \D t\\
		&I_{3,j}:=-4h^{3}\int\limits_{Q}\mbox{Re}\lb \Box v_{j}(t,x)\lb\omega\cdot\overline{\nabla_{x}v_{j}(t,x)}\rb\rb \D x \D t.
		\end{aligned}
		\end{align}
		We will estimate each of $I_{k,j}$ for $1\leq k\leq 3$ and each fixed $1\leq j\leq n$.
		We first simplify $I_{1,j}$. To estimate $I_{1,j}$, first consider the following integral for $0\leq s\leq T$
		\begin{align*}
		\begin{aligned}
		2\int\limits_{0}^{s}\int\limits_{\Omega}\left(\partial_{t}v_{j}(t,x)-\omega\cdot\n_{x}v_{j}(t,x)\right)v_{j}(t,x)\D x\D t=\int\limits_{\Omega}\lvert v_{j}(s,x)\rvert^{2}\D x-\int\limits_{0}^{s}\int\limits_{\Omega}\n_{x}\cdot\left(\lvert v_{j}(t,x)\rvert^{2}\omega\right)\D x\D t.
		\end{aligned}
		\end{align*}
		Now using Cauchy-Schwartz inequality on left hand side of the above equation and the fact that $v_{j}(t,x)|_{\Sigma}=0$, we have 
		\begin{align}\label{Equation after Cauchy Schwartz}
		\begin{aligned}
		\int\limits_{\Omega}\lvert v_{j}(s,x)\rvert^{2}\D x \leq \frac{1}{\epsilon^{2}}\int\limits_{0}^{s}\int\limits_{\Omega}\lvert\left(\partial_{t}-\omega\cdot\n_{x}\right)v_{j}(t,x)\rvert^{2}\D x\D t +\epsilon^{2}\int\limits_{0}^{s}\int\limits_{\Omega}\lvert v_{j}(t,x)\rvert^{2}\D x\D t
		\end{aligned}
		\end{align} 
		holds for any $\epsilon>0$. Now integrating both sides of \eqref{Equation after Cauchy Schwartz} with respect to  $s$ variable from $0$ to $T$, we have
		\begin{align*}
		\begin{aligned}
		\int\limits_{0}^{T}\int\limits_{\Omega}\lvert v_{j}(s,x)\rvert^{2}\D x \D s\leq \frac{T}{\epsilon^{2}}\int\limits_{0}^{T}\int\limits_{\Omega}\lvert\left(\partial_{t}-\omega\cdot\n_{x}\right)v_{j}(t,x)\rvert^{2}\D x\D t +T\epsilon^{2}\int\limits_{0}^{T}\int\limits_{\Omega}\lvert v_{j}(t,x)\rvert^{2}\D x\D t.
		\end{aligned} 
		\end{align*}
		Now choose $\epsilon>0$, small enough such that $1-T\epsilon^{2}>0$, we get 
		\begin{align}
		\begin{aligned}
		4Ch^{2}\int\limits_{0}^{T}\int\limits_{\Omega}\lvert v_{j}(t,x)\rvert^{2}\D x\D t \leq I_{1,j}
		\end{aligned}
		\end{align}
		where $C>0$ is some constant depending only on $T$.
		Next  using the integration by parts and the fact that 
		$v_{j}|_{\Sigma}=v_{j}|_{t=0}=\PD_{t}v_{j}|_{t=0}=0$,
		we have $I_{2,j}$ is 
		\begin{align*} 
		I_{2,j}&=4h^{3}\int\limits_{Q}\mbox{Re}\lb \Box v_{j}(t,x)\overline{\partial_{t}v_{j}(t,x)}\rb \D x \D t\\
		&=2h^{3}\int\limits_{Q} \frac{\PD}{\PD t} |\PD_{t} v_{j}(t,x)|^{2} \D x \D t -4h^{3}\int\limits_{Q}\mbox{Re}\lb\Delta v_{j}(t,x) \overline{\partial_{t}v_{j}(t,x)}\rb\D x \D t\\
		&= 2h^{3} \int\limits_{\Omega} \lb |\PD_{t}v_{j}(T,x)|^{2} +|\n_{x}v_{j}(T,x)|^{2} \rb \D x.
		\end{align*}
		Finally, we consider $I_{3,j}$. This is 
		\[
		I_{3,j}=-4h^{3} \int\limits_{Q}\mbox{Re}\lb \Box  v_{j}(t,x) \omega\cdot \overline{\n_{x}v_{j}(t,x)}\rb \D x \D t.
		\]
		We have 
		\begin{align*}
		I_{3,j}&=-4h^{3}\mbox{Re}\int\limits_{Q}\PD_{t}^{2}v_{j}(t,x)\overline{\omega\cdot \n_{x}v_{j}(t,x)} \D x \D t + 4h^{3}\mbox{Re}\int\limits_{Q}\Delta v_{j}(t,x) \overline{\omega\cdot \n_{x}v_{j}(t,x)} \D x \D t\\
		&=-4h^{3} \mbox{Re}\int\limits_{Q} \PD_{t}\lb \PD_{t} v_{j}(t,x)\overline{\omega\cdot \n_{x}v_{j}(t,x)} \rb \D x \D t+4h^{3}\mbox{Re}\int\limits_{Q} \PD_{t}v_{j}(t,x) \overline{\omega\cdot \n_{x}\PD_{t}v_{j}(t,x)} \D x \D t\\
		&\quad\quad +4h^{3}\mbox{Re}\int\limits_{Q}\n_{x}\cdot \lb \n_{x}v_{j}(t,x) \overline{\omega\cdot \n_{x} v_{j}(t,x)}\rb \D x \D t-4h^{3}\mbox{Re}\int\limits_{Q} \n_{x}v_{j}(t,x)\cdot \n_{x} \lb \overline{\omega\cdot \n_{x}v_{j}(t,x)}\rb \D x \D t\\
		&=-4h^{3}\mbox{Re}\int\limits_{\Omega} \PD_{t} v_{j}(T,x) \overline{\omega\cdot \n_{x}v_{j}(T,x)} \D x +2h^{3} \int\limits_{Q}\n_{x}\cdot \lb \omega |\PD_{t}v_{j}(t,x)|^{2}\rb \D x \D t\\
		&\quad \quad +2h^{3}\mbox{Re} \int\limits_{\Sigma} \PD_{\nu} v_{j}(t,x) \overline{\omega\cdot \n_{x}v_{j}(t,x)} \D S_{x} \D t -2h^{3} \int\limits_{Q}\n_{x} \cdot \lb \omega |\n_{x}v_{j}|^{2}\rb \D x \D t\\
		&=-4h^{3}\mbox{Re}\int\limits_{\Omega} \PD_{t} v_{j}(T,x) \overline{\omega\cdot \n_{x}v_{j}(T,x)} \D x  +2h^{3}\int\limits_{\Sigma} \omega \cdot \nu |\PD_{\nu}v_{j}|^{2} \D S_{x} \D t.
		\end{align*}
		In deriving the above equation, we used the fact that 
		\[
		2h^{3}\mbox{Re} \int\limits_{\Sigma} \PD_{\nu} v_{j}(t,x) \overline{\omega\cdot \n_{x}v_{j}(t,x)} \D S_{x} \D t= 2h^{3}\int\limits_{\Sigma} \omega\cdot \nu |\PD_{\nu} v_{j}|^{2} \D S_{x} \D t,
		\]
		since $v_{j}=0$ on $\Sigma$. Also note that $\partial_{t}v_{j}(t,x)=0$ and $|\n_{x} v_{j}|=|\PD_{\nu}v_{j}|$ on $\Sigma$.

		Therefore 
		\begin{align*}
		\int\limits_{Q} |\Box_{\vp}v_{j}(t,x)|^{2}\D x \D t &\geq 4Ch^{2}\int\limits_{0}^{T}\int\limits_{\Omega}\lvert v_{j}(t,x)\rvert^{2}+2h^{3} \int\limits_{\Omega} \lb |\PD_{t}v_{j}(T,x)|^{2} +|\n_{x}v_{j}(T,x)|^{2} \rb \D x\\
		&\quad\quad-4h^{3}\mbox{Re}\int\limits_{\Omega} \PD_{t} v_{j}(T,x) \overline{\omega\cdot \n_{x}v_{j}(T,x)} \D x +2h^{3}\int\limits_{\Sigma} \omega \cdot \nu |\PD_{\nu}v_{j}|^{2} \D S_{x} \D t.
		\end{align*}
		After using the Cauchy-Schwartz inequality to estimate third term, we get
		\begin{align}\label{Estimate for box phi}
		\begin{aligned}
		C\Big(h^{2}\int\limits_{Q}\lvert \vec{v}(t,x)\rvert^{2}+h^{3} \int\limits_{\Omega}  |\PD_{t}\vec{v}(T,x)|^{2}\D x -4h^{3}\int\limits_{\Omega} |\n_{x}\vec{v}(T,x)|^{2}  \D x\\+2h^{3}\int\limits_{\Sigma} \omega \cdot \nu |\PD_{\nu}\vec{v}|^{2} \D S_{x} \D t\Big)\leq C\int\limits_{Q}|\Box_{\vp}\vec{v}(t,x)|^{2}\D x \D t.
		\end{aligned}
		\end{align}
		Now we consider the conjugated operator $\Lc_{\vp}:=h^{2} e^{-\frac{\vp}{h}}\Lc_{q} e^{\frac{\vp}{h}}$. We have  
		\[
		\Lc_{\vp}\vec{v}(t,x)=h^{2}\lb e^{-\vp/h}\left(\Box+q\right) e^{\vp/h}\vec{v}(t,x)\rb=\Box_{\vp}\vec{v}(t,x)+h^{2}q(t,x)\vec{v}(t,x).
		\]
		By triangle inequality,
		\begin{equation}\label{Full conjugated}
		\int\limits_{Q}\left|\Lc_{\vp}\vec{v}(t,x)\right|^{2}\D x \D t\geq \frac{1}{2}\int\limits_{Q}|\Box_{\vp}\vec{v}(t,x)|^{2}\D x \D t-h^{4}\int\limits_{Q}|q(t,x)\vec{v}(t,x)|^{2}\D x \D t.
		\end{equation}
	We have 
		\begin{align*}
		&h^{4}\int\limits_{Q}\left|q(t,x)\vec{v}(t,x)\right|^{2}\D x \D t\leq Ch^{4}\int\limits_{Q}\lvert \vec{v}(t,x)\rvert^{2}\D x\D t
		\end{align*}
		where constant $C>0$ depends on $\lVert q\rVert_{L^{\infty}(Q)}$.
		Using this together with Equation \eqref{Estimate for box phi} in \eqref{Full conjugated}, we have that there exists a constant $C>0$ depending only on $T$, $\Omega$ and $q$ such that
		\begin{align*}
		\begin{aligned}
		C\left(h^{2}\int\limits_{Q}\lvert \vec{v}(t,x)\rvert^{2}+
		h^{3} \int\limits_{\Omega}  |\PD_{t}\vec{v}(T,x)|^{2}\D x +2h^{3}\int\limits_{\Sigma} \omega \cdot \nu |\PD_{\nu}\vec{v}|^{2} \D S_{x} \D t \right)\\\leq \int\limits_{Q}\left|\Lc_{\vp}\vec{v}(t,x)\right|^{2}\D x \D t +4h^{3}\int\limits_{\Omega} |\n_{x}\vec{v}(T,x)|^{2}  \D x
		\end{aligned}
		\end{align*}
		and this inequality holds for $h$ small enough. After dividing by $h^{2}$, we get
		\begin{align}\label{Estimate for conjugated operator final}
		\begin{aligned}
		C\left(\int\limits_{Q}\lvert \vec{v}(t,x)\rvert^{2}+
		h \int\limits_{\Omega}  |\PD_{t}\vec{v}(T,x)|^{2}\D x +2h\int\limits_{\Sigma} \omega \cdot \nu |\PD_{\nu}\vec{v}|^{2} \D S_{x} \D t \right)\\\leq \frac{1}{h^{2}}\int\limits_{Q}\left|\Lc_{\vp}\vec{v}(t,x)\right|^{2}\D x \D t +4h\int\limits_{\Omega} |\n_{x}\vec{v}(T,x)|^{2}  \D x.
		\end{aligned}
		\end{align}
		Let us now substitute $\vec{v}(t,x)=e^{-\frac{\vp}{h}} \vec{u}(t,x)$. We have 
		\begin{align*}
		&he^{-\vp/h}\partial_{t}u_{j}(t,x)=h\partial_{t}v_{j}+e^{-\vp/h}u_{j},\\
		&he^{-\vp/h}\nabla_{x}u_{j}=h\nabla_{x}v_{j}+e^{-\vp/h}\omega u_{j},\\
		&\partial_{\nu}v_{j}(t,x)|_{\Sigma}=e^{-\vp/h}\partial_{\nu}u_{j}|_{\Sigma}, \quad \mbox{since } u_{j}=0 \mbox{ on } \Sigma.
		\end{align*}
		Using the triangle inequality, we have 
		\begin{align*}
		\begin{aligned}
		&h\int\limits_{\Omega}e^{-2\vp(T,x)/h}\lvert\PD_{t} u_{j}(T,x)\rvert^{2}\D x-\frac{1}{h}\int\limits_{\Omega}e^{-2\vp(T,x)/h}\lvert u_{j}(T,x)\rvert^{2}\D x\leq C h\int\limits_{\Omega}e^{-2\vp(T,x)/h}\lvert\PD_{t} v_{j}(T,x)\rvert^{2}\D x\\
		&h\int\limits_{\Omega}\lvert\n_{x}v_{j}(T,x)\rvert^{2}\D x\leq  C\lb h\int\limits_{\Omega}e^{-2\vp(T,x)/h}\lvert\n_{x} u_{j}(T,x)\rvert^{2}\D x +\frac{1}{h}\int\limits_{\Omega}e^{-2\vp(T,x)/h}\lvert u_{j}(T,x)\rvert^{2}\D x\rb.
		\end{aligned}
		\end{align*}
		Using the above  inequalities and choosing $h$ small enough, we have 
		\begin{align*}
		\begin{aligned}
		&\int\limits_{Q}e^{-2\vp/h}\lvert \vec{u}(t,x)\rvert^{2}\D x\D t+h\int\limits_{\Omega}e^{-2\vp(T,x)/h}\lvert\PD_{t} u_{j}(T,x)\rvert^{2}\D x+2h\int\limits_{\Sigma}\omega\cdot\nu(x)e^{-2\vp/h}\lvert\vec{u}(t,x)\rvert^{2}\D S_{x}\D t\\
		&\leq C\left(h^{2}\int\limits_{Q}e^{-2\vp/h}\lvert\Lc_{q}\vec{u}(t,x)\rvert^{2}\D x\D t +h\int\limits_{\Omega}e^{-2\vp(T,x)/h}\lvert\n_{x}\vec{u}(T,x)\rvert^{2}\D x +\frac{1}{h}\int\limits_{\Omega}e^{-2\vp(T,x)/h}\lvert\vec{u}(T,x)\rvert^{2}\D x\right).
		\end{aligned}
		\end{align*}
		Finally,
		\begin{align*}
		\begin{aligned}
		&\lVert e^{-\vp/h}\vec{u}\rVert^{2}_{\Lbf(Q)}+h\lb  e^{-\vp/h}\partial_{\nu}\vp\partial_{\nu}\vec{u},e^{-\phi/h}\partial_{\nu}\vec{u}\rb_{\Lbf(\Sigma_{+,\omega})}+h\lb e^{-\vp(T,\cdot)/h}\partial_{t}\vec{u}(T,\cdot),e^{-\vp(T,\cdot)/h}\partial_{t}\vec{u}(T,\cdot)\rb_{\Lbf(\Omega)}\\
		&\leq C\Bigg(\lVert he^{-\vp/h}\Lc_{q} \vec{u}\rVert^{2}_{\Lbf(Q)}+\frac{1}{h}\lb e^{-\vp(T,\cdot)/h}u(T,\cdot),e^{-\vp(T,\cdot)/h}\vec{u}(T,\cdot)\rb_{\Lbf(\Omega)}\\
		&\quad \quad+h\lb e^{-\vp(T,\cdot)/h}\nabla_{x}\vec{u}(T,\cdot),e^{-\vp(T,\cdot)/h}\nabla_{x}\vec{u}(T,\cdot)\rb_{\Lbf(\Omega)}+h\lb  e^{-\vp/h}\lb-\partial_{\nu}\vp\rb\partial_{\nu}\vec{u},e^{-\vp/h}\partial_{\nu}\vec{u}\rb_{\Lbf(\Sigma_{-,\omega})}\Bigg).
		\end{aligned}
		\end{align*}
		This completes the proof.
	\end{proof}

	\section{Construction of Geometric optics solutions}\label{GO solutions}
	Aim of this section is to construct exponential growing and decaying solutions which will be used to prove the main result of this article. To  construct these solutions we follow very closely the ideas from \cite{Kian_Stability_potential_partial_data,Kian}  used for constructing the geometric optics solutions for the wave equation with a scalar potential. We state the following lemma which will be used for constructing the solutions. Proof of this is given in  \cite{Kian_Stability_potential_partial_data}.
	\begin{lemma}\label{Existence lemma decaying}\cite{Kian}
		Let  $\Box_{\pm\vp}$ be as defined in \eqref{Conjugated Box}, then for each $0<h<1$  there exists a bounded linear operator $\Box_{\pm\vp}^{*} : H^{1}(Q) \rightarrow H^{1}(Q)$ such that
		\begin{enumerate}
			\item $\Box_{\pm\vp}^{*}\lb\Box_{\pm\vp}f\rb=f, \ f \in H^{1}(Q)$
			\item $\lVert\Box^{*}_{\pm\vp}\rVert_{\Bc(L^{2}(Q))}\leq C$
			\item $\Box^{*}_{\pm\vp}\in\Bc\lb H^{1}(Q);H^{2}(Q)\rb \ \mbox{and}\ \lVert\Box^{*}_{\pm\vp}\rVert_{\Bc\lb H^{1}(Q);H^{2}(Q)\rb}\leq C$
		\end{enumerate}
	for some constant $C>0$ depending only on $Q$.
	\end{lemma}
	Using Lemma \ref{Existence lemma decaying} in the following Proposition, we  construct the exponential decaying solution for $\Lc_{q^*}v=0$. 
	\begin{proposition}\label{Decaying solution}
		 	Let $q$ and $\varphi$ be as in Theorem \ref{Boundary Carleman estimate}. Then, there exists an $h_{0}>0$ such that for all $0<h\leq h_{0}$, we can find $\vec{v}_d \in \Hbf^{2}(Q)$ satisfying $\Lc_{q^{*}}\vec{v}_d=0$ of the form 
			\begin{equation}\label{Go decaying}
			\vec{v}_{d}(t,x)= e^{-\frac{\vp}{h}}\lb \vec{B}_{d}(t,x) + h\vec{R}_{d}(t,x;h)\rb,
			\end{equation}
			where 
			\begin{equation}\label{Expression for Bd}
			\vec{B}_{d}(t,x)=e^{-i\zeta\cdot(t,x)}\vec{K}_{1}
			\end{equation}
			with $\zeta\in(1,-\omega)^{\perp}$, $\vec{K}_{1}$ is a constant $n$-vector and $\vec{R}_{d}\in \Hbf^{2}(Q)$ satisfies
			\begin{equation}\label{estimate for Remainder term1}
			\lVert \vec{R}_{d}\rVert_{\Lbf(Q)}\leq C .
			\end{equation}
				\begin{proof}
				We have \begin{align*}
				\begin{aligned}
				\Lc_{q^{*}}\vec{v}(t,x)
				=
				\begin{bmatrix}\vspace*{2mm}
				\Box{v}_{1}(t,x)+\sum_{j=1}^{n}\overline{q}_{j1}(t,x)v_{j}(t,x)\\
				\Box{v}_{2}(t,x)+\sum_{j=1}^{n}\overline{q}_{j2}(t,x)v_{j}(t,x)\\
				\vdots\\
				\Box{v}_{n}(t,x)+\sum_{j=1}^{n}\overline{q}_{jn}(t,x)v_{j}(t,x)
				\end{bmatrix}
				\end{aligned}
				\end{align*}
				and we are looking for $\vec{v}_{d}(t,x)$ of the form  \eqref{Go decaying} such that 
				\[\Lc_{q^{*}}\vec{v}_{d}(t,x)=0.\]
				Thus, we have 
				\begin{align}\label{Equation for vdi}
				\square v_{di}(t, x)+ \sum_{j=1}^n \overline{q}_{ji}(t,x) v_{dj}(t,x) &=0, \ \mbox{for $1 \leq i \leq n$}
				\end{align}
				where $v_{di}$ stands for the  $i$th component of $\vec{v}_{d}$. Also we denote  by $B_{di}$  and $R_{di}$ as the $i$th component of $\vec{B}_{d}$ and $\vec{R}_{d}$ respectively. Now using the expressions for $v_{di}$ from \eqref{Go decaying} in \eqref{Equation for vdi}, we have 
				\begin{align*}
				\begin{aligned}
				h^{2}\Box R_{di}-2h\lb \PD_{t}-\omega\cdot\n_{x}\rb R_{di}+h^{2}\sum_{j=1}^{n}\overline{q}_{ji}R_{dj}=-h\Box B_{di} -h\sum_{j=1}^{n}\overline{q}_{ji}B_{dj}
				\end{aligned}
				\end{align*}
				holds for $1\leq i\leq n$. Using Equation \eqref{Conjugated Box}, we have 
				\begin{align}\label{Equation for Rd}
				\begin{aligned}
				\Box_{-\vp}\vec{R}_{d}(t,x)=-h\Lc_{q^{*}}\vec{B}_{d}(t,x)-h^{2}q^{*}(t,x)\vec{R}_{d}(t,x).
				\end{aligned}
				\end{align} 
				Now for $\vec{w}\in \Hbf^{1}(Q)$, we define the map $\Fc: \Hbf^{1}(Q)\rightarrow \Hbf^{1}(Q)$ by 
				\[\Fc(\vec{w}):=\Box_{-\vp}^{*}\lb -h\Lc_{q^{*}}\vec{B}_{d}-h^{2}q^{*}\vec{w}\rb.\]
				which is well-defined from  Lemma \ref{Existence lemma decaying} and the fact that $q\in W^{1,\infty}(Q)$. Now using Lemma  \ref{Existence lemma decaying},  we have 
				\[ \lVert \Fc(\vec{w}_{1})-\Fc(\vec{w}_{2})\rVert_{\Hbf^{1}(Q)}=h^{2}\Big\lVert \Box_{-\vp}^{*}\lb q^{*}\left\{\vec{w}_{1}-\vec{w}_{2}\right\}\rb\Big\rVert_{\Hbf^{1}(Q)}\leq Ch^{2}\lVert \vec{w}_{1}-\vec{w}_{2}\rVert_{\Hbf^{1}(Q)}\]
				for some constant $C>0$ independent of $\vec{w}_{i}$ and $h$. Now choosing $h>0$ small enough such that $Ch^{2}<1$, we have by fixed point theorem, there exists $\vec{w}\in \Hbf^{1}(Q)$ such that $\Fc(\vec{w})=\vec{w}$. Now going back to Equation \eqref{Equation for Rd}  and using Lemma \ref{Existence lemma decaying},  we have $\vec{R}_{d}\in\Hbf^{2}(Q)$ and $\lVert\vec{R}_{d}\rVert_{\Lbf(Q)}\leq C$. 
				This completes the proof of Proposition \ref{Decaying solution}.
			\end{proof}
		\end{proposition}
	Next in the following proposition we construct the exponential growing solution to $\Lc_{q}\vec{v}=0$.
	\begin{proposition}\label{Growing solution}
		Let $q$ and $\varphi$ be as in Theorem \ref{Boundary Carleman estimate}. Then, there exists an $h_{0}>0$ such that for all $0<h\leq h_{0}$, we can find $\vec{v}_g \in \Hbf^{2}(Q)$ satisfying $\Lc_{q}\vec{v}_g=0$ of the form 
		\begin{equation}\label{Go growing}
		\vec{v}_{g}(t,x)= e^{\frac{\vp}{h}}\lb \vec{B}_{g}(t,x) + h\vec{R}_{g}(t,x;h)\rb,
		\end{equation}
	$\vec{B}_{g}(t,x):=\vec{K}_{2}$ is a constant $n$-vector and $\vec{R}_{g}\in \Hbf^{2}(Q)$ satisfies
		\begin{equation}\label{estimate for Remainder term2}
		\lVert \vec{R}_{g}\rVert_{\Lbf(Q)}\leq C .
		\end{equation}
	\begin{proof}
		Proof follows by using the similar arguments as used in proving Proposition \ref{Decaying solution}. 
		\end{proof}
	\end{proposition}

	\section{Recovery of $q$}\label{proof main result}
	In this section, we prove the  main Theorem \ref{Main Theorem} of this article.  The proof is based on deriving an integral identity followed by using the Carleman estimate and   geometric optic solutions constructed in \S \ref{GO solutions}, we conclude the proof of our main result.  
To derive the integral identity, let us consider $\vec{u}^{(j)}$ be the solutions to the following initial boundary value problems with matrix valued potential $q^{(j)}$ for $j=1,2$.


	\begin{align}\label{Equation for ui}
	\begin{aligned}
	\begin{cases}
	&\Lc_{q^{(j)}}\vec{u}^{(j)}(t,x)=0,\ (t,x)\in Q\\
	&\vec{u}^{(j)}(0,x)=\vec{\phi}(x),\ \partial_{t}\vec{u}^{(j)}(0,x)=\vec{\psi}(x), \ x\in\Omega\\
	&\vec{u}^{(j)}(t,x)=\vec{f}(t,x),\ (t,x)\in \Sigma.
	\end{cases}
	\end{aligned}
	\end{align} 
Also denote 
	\begin{align}\label{Difference defn}
	\begin{aligned}
	&\vec{u}(t,x):=\vec{u}^{(1)}(t,x)-\vec{u}^{(2)}(t,x)\\
	&q(t,x):=q^{(2)}(t,x)-q^{(1)}(t,x).
	\end{aligned}
	\end{align}
	Then $\vec{u}$ will satisfies the following initial boundary value problem: 
	\begin{align}\label{Equation for u linear}
	\begin{aligned}
	\begin{cases}
	&\Lc_{q^{(1)}}\vec{u}(t,x)
	=q(t,x)\vec{u}^{(2)}(t,x),\ (t,x)\in Q \\
	&\vec{u}(0,x)=\partial_{t}\vec{u}(0,x)=\vec{0},\ x\in\Omega\\
	&\vec{u}(t,x)=\vec{0},\ (t,x)\in \Sigma
	\end{cases}
	\end{aligned}
	\end{align}
	Let $\vec{v}(t,x)$ of the form given by \eqref{Go decaying} be the solution to following equation 
	\begin{align}\label{adjoint equation for u1}
	\Lc^{*}_{{q^{(1)}}}\vec{v}(t,x)
	=0\  \text{in}\ Q.
	\end{align}
	Also let $\vec{u}^{(2)}$ of the form given by  \eqref{Go growing} be solution to the following equation
	\begin{align}\label{Equation for u2}
	\begin{aligned}
	\begin{cases}
	&\Lc_{q^{(2)}}\vec{u}^{(2)}(t,x)=0,\ (t,x)\in Q\\
	&\vec{u}^{(2)}(0,x)=\vec{\phi}(x),\ \partial_{t}\vec{u}^{(2)}(0,x)=\vec{\psi}(x), \ x\in\Omega\\
	&\vec{u}^{(2)}(t,x)=\vec{f}(t,x),\ (t,x)\in \Sigma.
	\end{cases}
	\end{aligned}
	\end{align} 
	Using Theorem \ref{Exitence uniqueness theorem}, we have $\vec{u}\in \Hbf^{1}(Q)$ and $\partial_{\nu}\vec{u}\in \Lbf(\Sigma)$. 
Multiply  \eqref{Equation for u linear} by $\overline{\vec{v}(t,x)}\in \Hbf^{1}(Q)$ solution to \eqref{adjoint equation for u1} and integrate over $Q$. Now using integration by parts and taking into account the following: $\vec{u}|_{\Sigma}=\vec{0}$, $\vec{u}(T,x)=\vec{0}$, $\partial_{\nu}\vec{u}|_{G}=\vec{0}$, $\vec{u}|_{t=0}=\PD_{t}\vec{u}|_{t=0}=\vec{0}$ and $\Lc^{*}_{{q^{(1)}}}\vec{v}(t,x)=\vec{0}$ , we get 
	\begin{align}\label{Final integral identity}
	\int\limits_{Q}q(t,x)\vec{u}^{(2)}(t,x)\cdot\overline{\vec{v}(t,x)}\D x \D t&=\int\limits_{\Omega}\partial_{t}\vec{u}(T,x) \cdot\overline{\vec{v}(T,x)}\D x-\int\limits_{\Sigma\setminus{G}}\partial_{\nu}\vec{u}(t,x)\cdot\overline{\vec{v}(t,x)}\D S_{x}\D t.
	\end{align}
	\begin{lemma}
		Let  $\vec{u}^{(i)}$ for $i=1,2$ solutions to \eqref{Equation for ui} with $\vec{u}^{(2)}$ of the form \eqref{Go growing}. Let $\vec{u}(t,x)=\vec{u}^{(1)}(t,x)-\vec{u}^{(2)}(t,x)$, and $\vec{v}$ be of the form \eqref{Go decaying}. Then 
		\begin{equation}\label{first remainder}
		h\int\limits_{\Omega} \PD_{t}\vec{u}(T,x)\cdot \overline{\vec{v}(T,x)} \D x \to 0 \mbox{ as } h\to 0^{+}.
		\end{equation}
		\begin{equation}\label{second remainder}
		h\int\limits_{\Sigma\setminus{G}}\partial_{\nu}\vec{u}(t,x)\cdot\overline{\vec{v}(t,x)}\D S_{x}\D t \to 0 \mbox{ as } h\to 0^{+}.
		\end{equation}
	\end{lemma}
	\begin{proof}
		
		Using  \eqref{Go decaying}, \eqref{estimate for Remainder term1} and Cauchy-Schwartz inequality, we get 
		\begin{align*}
		&\left|h\int\limits_{\Omega}\partial_{t}\vec{u}(T,x)\cdot\overline{\vec{v}(T,x)}\D x\right|\leq \int\limits_{\Omega}h\left|\partial_{t}\vec{u}(T,x)\cdot e^{-\frac{\vp(T,x)}{h}}\overline{\left(\vec{B}_{d}(T,x)+h \vec{R}_{d}(T,x)\right)}\right|\D x \\
		&\leq C\left(\int\limits_{\Omega}h^{2}\left|\partial_{t}\vec{u}(T,x)e^{-\frac{\vp(T,x)}{h}}\right|^{2}\D x\right)^{\frac{1}{2}}\left(\int\limits_{\Omega}\left|e^{-i\xi\cdot(T,x)}\vec{K}_{1}+h\overline{\vec{R}_{d}(T,x)}\right|^{2}\D x\right)^{\frac{1}{2}}\\
		&\leq C\left(\int\limits_{\Omega}h^{2}\left|\partial_{t}\vec{u}(T,x)e^{-\frac{\vp(T,x)}{h}}\right|^{2}\D x\right)^{\frac{1}{2}}\left(1+\lVert h\vec{R}_{d}(T,\cdot)\rVert^{2}_{\Lbf(\Omega)}\right)^{\frac{1}{2}}\\
		&\leq C \left(\int\limits_{\Omega}h^{2}\left|\partial_{t}\vec{u}(T,x)e^{-\frac{\vp(T,x)}{h}}\right|^{2}\D x\right)^{\frac{1}{2}}.
		\end{align*}
		Now using the boundary Carleman estimate \eqref{Boundary Carleman estimate}, 	we get, 
		\begin{align*}
		h\int\limits_{\Omega}\left|\partial_{t}\vec{u}(T,x)e^{-\frac{\vp(T,x)}{h}}\right|^{2}\D x\leq C\lVert he^{-\vp/h} \Lc_{q^{(1)}} \vec{u}\rVert_{\Lbf(Q)}^{2}=C\lVert he^{-\vp/h}q \vec{u}^{(2)}\lVert^{2}_{\Lbf(Q)}.
		\end{align*}
Substituting \eqref{Go growing} for $\vec{u}^{(2)}$, we get,
		\[
		h\int\limits_{\Omega} \PD_{t}\vec{u}(T,x)\cdot \overline{\vec{v}(T,x)} \D x \to 0 \mbox{ as } h\to 0^{+}.
		\]
		For $\ve>0$, define 
		\[
		\PD\Omega_{+,\ve,\omega}=\{x\in \PD \Omega: \nu(x) \cdot \omega > \ve\},\ \mbox{and}\ 
		\Sigma_{+,\ve,\omega}=(0,T)\times \PD \Omega_{+,\ve,\omega}.
		\]
	Next we prove \eqref{second remainder}.  Since $\Sigma\setminus G\subseteq \Sigma_{+,\ve,\omega}$ for all $\omega$ such that $|\omega-\omega_{0}|\leq \ve$, substituting $\vec{v}=\vec{v}_{d}$ from \eqref{Go decaying} in \eqref{second remainder} we have 
		\begin{align*}
		\begin{aligned}
		&\left|\int\limits_{\Sigma\setminus{G}}\partial_{\nu}\vec{u}(t,x)\cdot \overline{\vec{v}(t,x)}\D S_{x}\D t\right|\leq\int\limits_{\Sigma{+,\ve,\omega}}\left|\partial_{\nu}\vec{u}(t,x)\cdot e^{-\frac{\vp}{h}}\left(\vec{B}_{d}+h\vec{R}_{d}\right)(t,x)\right|\D S_{x}\D t\\
		&\leq C\left(1+\lVert h\vec{R}_{d}\rVert_{L^{2}(\Sigma)}^{2}\right)^{\frac{1}{2}}\left(\int\limits_{\Sigma{+,\ve,\omega}}\left|\partial_{\nu}\vec{u}(t,x)e^{-\frac{\vp}{h}}\right|^{2}\D S_{x}\D t\right)
		\end{aligned}
		\end{align*}
		with $C>0$ is independent of $h$ and this inequality holds for all $\omega $ such that $|\omega -\omega_{0}|\leq \ve$. Now using trace theorem, we have that  $\lVert \vec{R}_{d}\rVert_{\Lbf(\Sigma)}\leq C\lVert \vec{R}_{d}\rVert_{\Hbf^{1}(Q)}$.
	Using this, we get 
		\begin{align*}
		\left|\int\limits_{\Sigma\setminus{G}}\partial_{\nu}\vec{u}(t,x)\cdot \vec{v}(t,x)\D S_{x}\D t\right|\leq C\left(\int\limits_{\Sigma{+,\ve,\omega}}\left|\partial_{\nu}\vec{u}(t,x)e^{-\frac{\vp}{h}}\right|^{2}\D S_{x}\D t\right)^{\frac{1}{2}}.
		\end{align*}
		Now 
		\begin{align*}
		\int\limits_{\Sigma_{+},\ve,\omega}\left|\partial_{\nu}\vec{u}(t,x)e^{-\frac{\vp}{h}}\right|^{2}\D S_{x}\D t& =\frac{1}{\ve}\int \limits_{\Sigma_{+},\ve,\omega}\ve \left|\partial_{\nu}\vec{u}(t,x)e^{-\frac{\vp}{h}}\right|^{2}\D S_{x}\D t\\
		&\leq \frac{1}{\ve}\int\limits_{\Sigma_{+},\ve,\omega} \PD_{\nu}\vp \left|\partial_{\nu}\vec{u}(t,x)e^{-\frac{\vp}{h}}\right|^{2}\D S_{x}\D t.
		\end{align*}
		Using  \eqref{Boundary carleman estimate}, we have 
		\[
		\frac{h}{\ve}\int\limits_{\Sigma_{+},\ve,\omega} \PD_{\nu}\vp \left|\partial_{\nu}\vec{u}(t,x)e^{-\frac{\vp}{h}}\right|^{2}\D S_{x}\D t \leq C\lVert he^{-\vp/h} \Lc_{q^{(1)}} \vec{u}\rVert_{\Lbf(Q)}^{2}.
		\]
	Now proceeding as before, we get
		\[
		h\int\limits_{\Sigma\setminus{G}}\partial_{\nu}\vec{u}(t,x)\cdot\overline{\vec{v}(t,x)}\D S_{x}\D t \to 0 \mbox{ as } h\to 0^{+}.
		\]
	\end{proof} 
\noindent	Substituting \eqref{Go growing} for $\vec{u}^{(2)}$ and \eqref{Go decaying} for $\vec{v}$ in \eqref{Final integral identity} and using \eqref{first remainder} and \eqref{second remainder}, we get
	\[
	\int\limits_{\Rb^{1+n}} e^{-\I \xi \cdot (t,x)} q(t,x)\vec{K}_{1}\cdot \vec{K}_{2}\D x \D t=0, \mbox{ for} \ \xi \in (1,-\omega)^{\perp}, \ \mbox{for constant vectors} \ \vec{K}_{1}, \ \vec{K}_{2} \mbox{ and } \omega \mbox{ near } \omega_{0}.
	\]

	The set of all $\xi$ such that $\xi\in (1,-\omega)^{\perp}$ for $\omega$ near $\omega_{0}$ forms an open cone and since $q\in W^{1,\infty}(Q)$ has compact support therefore using the Paley-Wiener theorem we conclude that $q(t,x)\vec{K_{1}}\cdot \vec{K}_{2}=0$ for all $(t,x)\in Q$ and arbitrary constant vector $\vec{K}_{1}$ and $\vec{K}_{2}$. Thus, we have $q_{1}(t,x)=q_{2}(t,x)$. This completes the proof of Theorem \ref{Main Theorem}.

	\section*{Acknowledgments}
The work of second author is supported by NSAF grant (No. U1930402).


\begin{thebibliography}{99}
		\bibitem{Anikonov_Cheng_Yamamoto} Y. E. Anikonov, J. Cheng and M. Yamamoto; A uniqueness result in an inverse hyperbolic problem with analyticity, European J. Appl. Math. 15 (2004), no. 5, 533–-543.
		\bibitem{Avdonin_Belishev_Imanov}
S.A. Avdonin, M. I.  Belishev and S. A. Ivanov;
Boundary control and an inverse matrix problem for the equation $u_{tt}-u_{xx}+V(x)u=0.$ (Russian)
Mat. Sb. 182 (1991), no. 3, 307-331; translation in
Math. USSR-Sb. 72 (1992), no. 2, 287-310.
	
		

		\bibitem{Belishev_BC_Method_2011}
		M.I. Belishev;
		\newblock Boundary control method in dynamical inverse problems---an
		introductory course.
		\newblock In {\em Dynamical inverse problems: theory and application}, volume
		529 of {\em CISM Courses and Lect.}, pages 85--150. SpringerWienNewYork,
		Vienna, 2011.
		
			\bibitem{Bellassoued_Jellali_Yamamoto_Lipschitz_stability_hyperbolic}
		\newblock
		{ M. Bellassoued, D. Jellali and M. Yamamoto;}\newblock{ Lipschitz stability for a hyperbolic inverse problem by finite local boundary data, Appl. Anal. 85 (2006), no. 10, 1219–1243.}
			\bibitem{Bellassoued_Yamamoto} M. Bellassoued and M. Yamamoto; Determination of a coefficient in the wave equation with a single measurement, Appl. Anal. 87 (2008), no. 8, 901--920.
		
		\bibitem{Bellassoued_Jellali_Yamamoto_stability_hyperbolic}
		\newblock {M. Bellassoued, D. Jellali and M. Yamamoto;}  \newblock
		Stability estimate for the hyperbolic inverse boundary value problem by local Dirichlet-to-Neumann map, J. Math. Anal. Appl. 343 (2008), no. 2, 1036–1046.
			\bibitem{Bellassoued_Rassas} M. Bellassoued and I. Rassas; Stability estimate in the determination of a time-dependent
		coefficient for hyperbolic equation by partial
		Dirichlet-to-Neumann map, Appl. Anal. 98 (2019), no. 15, 2751-2782.
		
			\bibitem{Ibtissem_Stability_potential}	I. Ben A\"icha; Stability estimate for hyperbolic inverse problem with time-dependent coefficient, Inverse Problems, 31
		(2015), 125010.
		\bibitem{Blagoveshchenskii} A. S. Blagoveshchenskii; On a nonselfadjoint inverse boundary-value problem in matrix form for a hyperbolic differential equation, Problemy Mat. Fiz., vyp. 5, Izdat. Leningrad. Gos. Univ., Leningrad, 1971, pp. 38-61; English transl. in Topics in Math. Phys., no. 5, Plenum Press, New York, 1972.
	
		\bibitem{Bukhgeuim_Klibanov_Uniqueness_1981}
		A.L. Bukhge\u{i}m and M.V. Klibanov;
		\newblock Uniqueness in the large of a class of multidimensional inverse
		problems,
		\newblock {\em Dokl. Akad. Nauk SSSR}, 260(2):269--272, 1981.
		
		
		\bibitem{Bukhgeim_Uhlmann_Calderon_problem_partial_Cauchy_data_2002}
		A.L. Bukhge\u{i}m and G. Uhlmann;
		\newblock Recovering a potential from partial {C}auchy data, Comm Partial Differential Equ.
		2002; 27(3–4):653-668.
		\bibitem{Eskin_Ralston_Yang_Mills}
	G.	Eskin and J. Ralston;
Inverse scattering problems for the Schrödinger operators with external Yang-Mills potentials. Partial differential equations and their applications (Toronto, ON, 1995), 91--106,
CRM Proc. Lecture Notes, 12, Amer. Math. Soc., Providence, RI, 1997.
		\bibitem{Eskin_Ralston}
		G. Eskin and J. Ralston;
Inverse boundary value problems for systems of partial differential equations. (English summary) Recent development in theories \& numerics, 105--113, World Sci. Publ., River Edge, NJ, 2003.
35R30 (35J10)
	
	
		\bibitem{Cipolatti_Yamamoto} R. Cipolatti and M. Yamamoto; An inverse problem for a wave equation with arbitrary initial values and a finite time of observations. Inverse Problems 27 (2011), no. 9, 095006, 15 pp.

		
		\bibitem{Eskin_New_Approach_local_2006}
		G. Eskin;
		\newblock A new approach to hyperbolic inverse problems,
		\newblock {\em Inverse Problems}, 22(3):815--831, 2006.
		
		\bibitem{Eskin_IHP_time-dependent_2007}
		G. Eskin;
		\newblock Inverse hyperbolic problems with time-dependent coefficients,
		\newblock {\em Comm. Partial Differential Equations}, 32(10-12):1737--1758,
		2007.
		
		\bibitem{Eskin_New_Approcah_global_2007}
		G. Eskin;
		\newblock A new approach to hyperbolic inverse problems. {II}. {G}lobal step.
		\newblock {\em Inverse Problems}, 23(6):2343--2356, 2007.
				\bibitem{Khanfer_Bukhgeim}A. Khanfer and A. Bukhge\u{i}m; Inverse problem for one-dimensional wave equation with matrix potential, J. Inverse Ill-Posed Probl. 2019; 27(2): 217-223.
			\bibitem{Hu_Kian_Singular_potential_determination}	G. Hu and Y. Kian; Determination of singular time-dependent coefficients for wave equations from full and partial data,
		Inverse Probl. Imaging, 12 (2018), 745-772.
		
		\bibitem{Hussein_Lesnic_Yamamoto} S.O. Hussein, D. Lesnic and M. Yamamoto; Reconstruction of space-dependent potential and/or damping coefficients in the wave equation, Comput. Math. Appl. 74 (2017), no. 6, 1435–1454.
		\bibitem{Imanuvilov_Yamamoto} O. Yu. Imanuvilov and M. Yamamoto; Global uniqueness and stability in determining coefficients of wave equations, Comm. Partial Differential Equations 26 (2001), no. 7-8, 1409--1425.
		\bibitem{Isakov_Completeness_Product_solutions_IP_1991}
		V. Isakov;
		\newblock Completeness of products of solutions and some inverse problems for
		{PDE},
		\newblock {\em J. Differential Equations}, 92(2):305--316, 1991.
		
		\bibitem{Isakov_IP_hyperbolic_Damping_Potential_time-independent_1991}
		V. Isakov;
		\newblock An inverse hyperbolic problem with many boundary measurements,
		\newblock {\em Comm. Partial Differential Equations}, 16(6-7):1183--1195, 1991.
		
	
		
		\bibitem{Katchalov_Kurylev_Lassas_Book_2001}
		A. Katchalov, Y. Kurylev, and M. Lassas;
		\newblock {\em Inverse boundary spectral problems}, volume 123 of {\em Chapman
			\& Hall/CRC Monographs and Surveys in Pure and Applied Mathematics},
		\newblock Chapman \& Hall/CRC, Boca Raton, FL, 2001.
		

		
		\bibitem{Kian_Stability_potential_partial_data}Y. Kian; Stability in the determination of a time-dependent coefficient for wave equations from
		partial data, J. Math. Anal. Appl., 436 (2016), pp. 408–428.
		\bibitem{Kian}
		Y. Kian;
		\newblock Unique determination of a time-dependent potential for wave equations from partial data,
	\newblock {\em Ann. Inst. H. Poincar\'{e} Anal. Non Lin\'{e}aire}, 34 (2017) 973-990.	
		
		\bibitem{Kian_Damping_partial_data_2016}
		Y. Kian;
		\newblock Recovery of time-dependent damping coefficients and potentials
		appearing in wave equations from partial data,
		\newblock {\em SIAM J. Math. Anal.}, 48(6):4021--4046, 2016.
		
				\bibitem{Kian_Oksanen_anisotropic_wave_potential}Y. Kian and L. Oksanen; Recovery of time-dependent coefficient on Riemanian manifold for hyperbolic equations, International Mathematics Research Notices, Volume 2019, Issue 16, August 2019, Pages 5087-5126, https://doi.org/10.1093/imrn/rnx263.
		
			\bibitem{Krishnan_Vashisth_Relativistic}
		V. P. Krishnan and M. Vashisth; An inverse problem for the relativistic Schrödinger equation with	partial boundary data, Applicable Analysis, doi.org/10.1080/00036811.2018.1549321.
				
				\bibitem{Lions_Magenese_Book} J.L. Lions and E. Magenes; Non-Homogeneous Boundary Value Problems and Applications, Vol. I, Dunod, Paris, 1968.
				
				\bibitem{Lasiecka_Lions_Triggiani_Nonhomogeneous_BVP_hyperbolic_1986}
				I. Lasiecka, J.L. Lions and R. Triggiani;
				\newblock Nonhomogeneous boundary value problems for second order hyperbolic
				operators,
				\newblock {\em J. Math. Pures Appl. (9)}, 65(2):149--192, 1986.
		
		\bibitem{Lions}
		J.L. Lions;
		\'{E}quations diff\'{e}rentielles op\'{e}rationnelles et probl\`emes aux
   limites (French)
   \newblock Die Grundlehren der mathematischen Wissenschaften, Bd. 111,
\newblock Springer-Verlag, Berlin-G\"{o}ttingen-Heidelberg 1961 ix+292 pp.


		
		
		


		\bibitem{Rakesh_Symes_Uniqueness_1988}
		Rakesh and W.W. Symes;
		\newblock Uniqueness for an inverse problem for the wave equation,
		\newblock {\em Comm. Partial Differential Equations}, 13(1):87--96, 1988.
		
		\bibitem{Ramm_Rakesh_Property_C_1991}
		A.G. Ramm and Rakesh;
		\newblock Property {$C$} and an inverse problem for a hyperbolic equation,
		\newblock {\em J. Math. Anal. Appl.}, 156(1):209--219, 1991.
		
		\bibitem{Ramm_Sjostrand_IP_wave_equation_potential_1991}
		A.G. Ramm and J. Sj\"ostrand;
		\newblock An inverse problem of the wave equation,
		\newblock {\em Math. Z.}, 206(1):119--130, 1991.
		
		\bibitem{Salazar_time-dependent_first_order_perturbation_2013}
		R. Salazar;
		\newblock Determination of time-dependent coefficients for a hyperbolic inverse
		problem,
		\newblock {\em Inverse Problems}, 29(9):095015, 17, 2013.
		
	
		
		\bibitem{Stefanov-Yang}
		P. D. Stefanov and Y. Yang;
		\newblock The inverse problem for the {D}irichlet-to-{N}eumann map on
		{L}orentzian manifolds,
		\newblock {\em Anal. PDE}, 11(6):1381--1414, 2018.
		
		\bibitem{Stefanov_Inverse_scattering_potential_time_dependent_1989}
		P.D. Stefanov;
		\newblock Inverse scattering problem for the wave equation with time-dependent
		potential,
		\newblock {\em J. Math. Anal. Appl.}, 140(2):351--362, 1989.
		

		
		\bibitem{Sylvester_Uhlmann_Calderon_problem_1987}
		J. Sylvester and G. Uhlmann;
		\newblock A global uniqueness theorem for an inverse boundary value problem,
		\newblock {\em Ann. of Math. (2)}, 125(1):153--169, 1987.
		
	
		
		
		
		



	\end{thebibliography}
\end{document}